\documentclass[10pt]{amsart}
\usepackage[all]{xy}
\usepackage{amsfonts,amssymb,amscd,amsmath,enumerate,verbatim,calc}
\usepackage{mathrsfs}
\numberwithin{equation}{section}
\newtheorem{theorem}{Theorem}[section]

\newtheorem{prop}[theorem]{Proposition}
\newtheorem{defn}[theorem]{Definition}

\newtheorem{ex}[theorem]{Example}
\newtheorem{re}[theorem]{Remark}

\begin{document}

\title[Formal deformations and extensions of `twisted' Lie algebras]
 {Formal deformations and extensions of `twisted' Lie algebras}

\bibliographystyle{amsplain}

\author[I. Basdouri, E. Peyghan, M.A. Sadraoui, R.Saha]{Imed Basdouri, Esmael Peyghan, Mohamed Amin Sadraoui, Ripan Saha}
\address{University of Gafsa, Faculty of Sciences Gafsa, 2112 Gafsa, Tunisia.}
\email{basdourimed@yahoo.fr}
\address{Department of Mathematics, Faculty of Science, Arak University,
	Arak, 38156-8-8349, Iran.}
\email{e-peyghan@araku.ac.ir}
\address{University of Sfax, Faculty of Sciences Sfax, BP
	1171, 3038 Sfax, Tunisia.}
\email{aminsadrawi@gmail.com}
\address{Department of Mathematics, Raiganj University, Raiganj 733134, West Bengal, India}
\email{ripanjumaths@gmail.com}




\begin{abstract}
The interplay between derivations and algebraic structures has been a subject of significant interest and exploration.  Inspired by Yau's twist and the Leibniz rule, we investigate the formal deformation of twisted Lie algebras by invertible derivations, herein referred to as ``InvDer Lie". We define representations of InvDer Lie, elucidate cohomology structures of order 1 and 2, and identify infinitesimals as 2-cocycles. Furthermore, we explore central extensions of InvDer Lie, revealing their intricate relationship with cohomology theory.
\end{abstract}

\maketitle

\section{Introduction}
In the realm of Lie algebras, a field rich with structural intricacies and profound mathematical implications, the notion of derivations has long been a focal point of investigation. A derivation on a Lie algebra \( (L, [\cdot, \cdot]_\mathrm{L}) \) offers insight into how the algebraic structure morphs under certain linear maps. 

Building upon this foundation, recent research, as articulated by Y. Donald in \cite{Y0}, delves into the concept of twisted Hom algebras. Here, the aim is to construct a new Hom algebra \( (L, \mu_\alpha, \alpha) \) from an existing algebra \( (L, \mu) \), where \( \mu_\alpha = \alpha \circ \mu \). This endeavor, detailed in Theorem 2.4 of \cite{Y0}, opens avenues for exploring alternative algebraic structures and their interplay with derivations.

Derivations themselves hold a pivotal role within Lie algebra theory. A linear map \( \delta: L \rightarrow L \) is deemed a derivation on \( (L, [\cdot, \cdot]_\mathrm{L}) \) if it adheres to the Leibniz rule, succinctly expressed as \( \delta([x, y]_\mathrm{L}) = [\delta x, y]_\mathrm{L} + [x, \delta y]_\mathrm{L} \) for all \( x, y \in L \). 

The implications of derivations extend far beyond mere manipulation of algebraic elements. In \cite{T0}, for instance, LieDer pairs are introduced, offering a novel perspective on algebraic constructions through derivations. Similarly, \cite{C0} offers a comprehensive characterization of the space of all derivations of Lie algebras, denoted as \( \text{Der}(L) \). In \cite{G0}, the authors studied $3$-Lie algebras with derivations from the cohomological perspective. Also see \cite{B1,B2}, where the same concept is studied in more details.

Motivated by the profound implications of Yau's twist \cite{Y0} and the insights gleaned from the Leibniz rule, a natural progression emerges: the exploration of twisted Lie algebras through the lens of invertible derivations. While the straightforward application of the twist, given by Yau, fails to preserve the Jacobi identity. In \cite{B0}, the authors endeavors to augment this framework by imposing conditions on the derivation \( \delta \) to ensure the preservation of this fundamental algebraic property.

In this paper, we define representations of InvDer Lie algebra, and introduce cohomology of order $1$ and $2$. Following Gerstenhaber \cite{G63, G64} classical formal deformation theory, we study formal deformations of InvDer Lie algebra and the identification of infinitesimals as 2-cocycles. Furthermore, we delve into the realm of central extensions of InvDer Lie algebra, elucidating their relationship with cohomology theory and unveiling the intricate interplay between algebraic structures and geometric concepts.

Through this endeavor, we aim to contribute to the rich tapestry of Lie algebra theory, offering novel insights into the interplay between derivations, twists, and fundamental algebraic properties, with potential implications spanning diverse fields of mathematics and theoretical physics.

Throughout this paper we consider $\mathbb{K}$ as a field of characteristic $0$ and all algebras are defined over $\mathbb{K}$.
	\section{Preliminary}
In this section, we recall some basic definitions and results which we shall use throughout the paper.
\begin{defn}
	Let $(\mathfrak{L}=(L,[-,-]_\mathrm{L}))$ be a Lie algebra. A linear map $\delta:L\rightarrow L$ is called a derivation on $\mathfrak{L}$ if 
	\begin{equation*}
		\delta[x,y]_\mathrm{L}=[\delta x,\delta y]_\mathrm{L}+[x,\delta y]_\mathrm{L},\quad ~\text{for all}~ x,y\in L.
	\end{equation*}
\end{defn}
We denote the set of all derivations of $\mathfrak{L}$ by $\mathbf{\mathrm{Der(L)}}$.
\begin{prop}\cite{B0}
	Let $\mathfrak{L}$ be a Lie algebra and $\delta$ be an invertible derivation. Then, we have
	\begin{equation}\label{Invderivation}
		[\delta x,\delta y]_L=\delta^2[x,y]_L ~\text{if and only if}~ \delta^{-1} \in \mathrm{Der(L)},\quad \text{for all} x,y\in L.
	\end{equation}
\end{prop}
Throughout the paper, we denote the set of all invertible derivation on $\mathfrak{L}$ satisfying equation \eqref{Invderivation} by $\mathbf{\mathrm{InvDer(L)}}$, that is, 
\begin{equation*}
	\mathrm{\textbf{InvDer(L)}}:=\{\delta:L\rightarrow L ~\mid~ \delta~ \text{is an invertible derivation, and}~ [\delta x,\delta y]_L=\delta^2[x,y]_L ~\text{for all}~ x,y\in L\}.
\end{equation*}
If $\mathrm{\delta\in InvDer(L)}$ then we say that $\delta$ is an \textbf{Inv-derivation} of $\mathfrak{L}$.\\
In the next theorem, inspired by \cite{Y0}, the authors in \cite{B0} present a Lie algebra structure by an Inv-derivation. 
\begin{theorem}\cite{B0}
	Let $(L,[-,-]_\mathrm{L})$ be a Lie algebra and $\delta:L \rightarrow L$ be an Inv-derivation. Then $(L,[-,-]_{\delta}=\delta \circ [-,-]_\mathrm{L})$ is a Lie algebra, which is called \textbf{ twisted Lie algebra by an invertible derivation }. 
\end{theorem}
\begin{prop} \cite{B0} \label{prop-2.4}
	Let $(L,[-,-]_\mathrm{L})$ be a Lie algebra and $\delta$ be an Inv-derivation on $L$. Then, the following equation holds
	\begin{equation*}
		\circlearrowleft_{x,y,z}[x,[y,z]_{\delta}]_\mathrm{L}=	\circlearrowleft_{x,y,z}[\delta x,[y,z]_\mathrm{L}]_\mathrm{L},\ \ \ \forall x,y,z \in L.
	\end{equation*}	
\end{prop}

\begin{re}\label{remark-2.5}
Suppose that \(\delta\) is invertible, and
\[
\circlearrowleft_{x,y,z} [x, [y, z]_\delta]_\delta = 0.
\]
Consequently,
\[
\delta \left( \circlearrowleft_{x,y,z} [x, [y, z]_\delta] \right) = 0.
\]
Since \(\ker(\delta) = 0\), it follows that
\[
\circlearrowleft_{x,y,z} [x, [y, z]_\delta] = 0.
\]
\end{re}

\begin{defn} \cite{B0}\label{defn-twisted}
	An InvDer Lie algebra is a Lie algebra $\mathfrak{L}$ equipped with an Inv-derivation $\delta$ satisfying the \textbf{InvDer-Jacobi} identity as follows 
	\begin{equation}\label{Inv Jacobi}
		\mathrm{\circlearrowleft_{x,y,z}[\delta x,[y,z]_L]_L=0,\quad \text{for all}~ x,y,z\in L.}
	\end{equation}
\end{defn}

According to the Proposition \ref{prop-2.4} and \ref{remark-2.5}, InvDer-Jacobi idenity is redundant. Therefore, we can rewrite the Definition \ref{defn-twisted} as follows:

\begin{defn}
An InvDer Lie algebra is a Lie algebra $\mathfrak{L}$ equipped with an Inv-derivation $\delta$.
\end{defn}

We denote it by $\mathrm{(L,[-,-]_L,\delta)}$ or simply by $(\mathfrak{L},\delta)$ if there is no confusion.
\section{Representation of InvDer Lie algebras}
\label{section representation}
In this section, we study InvDer Lie algebraic structures in a detailed way, where we give some basic results.
\begin{defn}
	Let $(L_1,[-,-]_1,\delta_1)$ and $(L_2,[-,-]_2,\delta_2)$ be two InvDer Lie algebras. A linear map $\varphi : L_1 \rightarrow L_2$ is said to be an \textbf{InvDer Lie algebra homomorphism} if it is a Lie algebra homomorphism from $L_1$ to $L_2$ satisfying the following identity 
	\begin{equation}\label{eqt InvDer Lie morphism}
		\varphi \circ \delta_1=\delta_2 \circ \varphi.
	\end{equation}
\end{defn}
Recall that a Lie algebra homomorphism $\varphi$ from $(L_1,[\cdot,\cdot]_1)$ to $(L_2,[\cdot,\cdot]_2)$ is a linear map satisfying
\begin{equation*}
	\varphi ([x,y]_1)=[\varphi(x),\varphi(y)]_2,\ \ \ x,y \in L_1.
\end{equation*} 
\begin{defn}
	A \textbf{representation of InvDer Lie algebra} $(\mathfrak{L},\delta)$ on a vector space $V$ with respect to $\delta_V \in \mathrm{gl}(V)$, where $\delta_\mathrm{V}$ is invertible, is a linear map $\rho:L \rightarrow \mathrm{gl}(V)$ such that for all $x,y,\in L$, we have 
	\begin{eqnarray}
		\rho(\delta x) \circ \delta_V&=&\delta_V^2 \circ \rho(x), \label{eq rep 1} \\
		\rho([x,y]_\mathrm{L}) \circ \delta_V&=&\rho(\delta x) \circ \rho(y)-\rho(\delta y)\circ \rho(x), \label{eq rep 2}\\
		\delta_V \circ \rho(x)&=&\rho(\delta x)+\rho(x) \circ \delta_V. \label{eq rep 3}
	\end{eqnarray}
	Such a representation is denoted by $(V;\rho,\delta_V)=(\mathcal{V},\delta_V)$.
\end{defn}
\begin{ex} \label{ex adj rep}
	The adjoint representation of InvDer Lie algebra $\mathfrak{L}$, denoted by $\mathrm{ad}$, is defined by 
	\begin{equation*}
	\left\{
	\begin{array}{cc}
	&\hspace{-3cm}\mathrm{ad}: L \rightarrow L,\\
	&\hspace{-.2cm}\mathrm{ad}(x)(y)=[x,y]_\mathrm{L},\ \ \  \forall x,y \in L.
	\end{array}
	\right.
	\end{equation*}
	With the above notations, we have 
	\begin{align*}
		\mathrm{ad}(\delta x) \circ \delta&=\delta^2 \circ \mathrm{ad}(x), \\
		\mathrm{ad}([x,y]_\mathrm{L}) \circ \delta&=\mathrm{ad}(\delta x) \circ \mathrm{ad}(y)-\mathrm{ad}(\delta y) \circ \mathrm{ad}(x), \\
		\delta \circ \mathrm{ad}(x)&=\mathrm{ad}(\delta x)+\mathrm{ad}(x) \circ \delta,
	\end{align*}
	which is denoted by $(L;\mathrm{ad},\delta)$.
\end{ex}
Here, $\delta_V \in \mathrm{gl}(V)$ represents a linear transformation of $V$. What would occur if $\delta_V$ were to become a derivation, and furthermore, an Inv-derivation (implying that $\delta_V \in \mathrm{InvDer}(\mathrm{gl}(V))$)? \\
This is what we will investigate in the next proposition.
\begin{prop}
	Let $(V;\rho,\delta_V)$ be a representation of the InvDer Lie algebra $(\mathfrak{L},\delta)$. Then the following two identities holds, for all $x,y \in L$
	\begin{eqnarray}
		\delta_V \in \mathrm{Der(gl}(V)) &\Leftrightarrow& \rho(x)(\delta_V \rho(y))=\rho(y)(\delta_V \rho(x)) \ , \ \label{eq Der gl 1}\\
		\delta_V \in \mathrm{InvDer(gl}(V)) &\Leftrightarrow& \rho(x)(\delta_V^2 \rho(y))=\rho(y)(\delta_V^2 \rho(x)). \label{eq Der gl 2}
	\end{eqnarray}
\end{prop}
\begin{proof}
	Let $x,y \in L$. For the first identity, we have
	\begin{align*}
		\delta_V[\rho(x),\rho(y)]&=\delta_V \big(\rho(x)\rho(y)-\rho(y)\rho(x) \big) \\
		&=\delta_V\big(\rho(x)\rho(y)\big)-\delta_V\big(\rho(y)\rho(x)\big)\\
		&=(\delta_V\rho(x))\rho(y)-(\delta_V\rho(y))\rho(x) \\
		&=\big(\rho(\delta x)+\rho(x)\delta_V \big)\rho(y)-\big(\rho(\delta y)+\rho(y)\delta_V \big)\rho(x)\\
		&=\rho(\delta x)\rho(y)+(\rho(x)\delta_V)\rho(y)-\rho(\delta y)\rho(x)-(\rho(y)\delta_V)\rho(x).
	\end{align*}
	On the other side, we have
	\begin{align*}
		[\delta_V\rho(x),\rho(y)]&=(\delta_V\rho(x))\rho(y)-\rho(y)(\delta_V\rho(x)) \\
		&=\rho(\delta x)\rho(y)+(\rho(x)\delta_V)\rho(y)-\rho(y)\rho(\delta x)-\rho(y)(\rho(x)\delta_V).
		\end{align*}
Similarly, we have
\begin{align*}
		[\rho(x),\delta_V\rho(y)]&=\rho(x)\rho(\delta y)+\rho(x)(\rho(y)\delta_V)-\rho(\delta y)\rho(x)-(\rho(y)\delta_V)\rho(x).
	\end{align*}
	By computing the sum of the two last terms we obtain
	\begin{equation*}
		[\delta_V\rho(x),\rho(y)]+[\rho(x),\delta_V\rho(y)]=	\delta_V[\rho(x),\rho(y)]+\big(\rho(x)\rho(\delta y)+\rho(x)(\rho(y)\delta_V)-\rho(y)\rho(\delta x)-\rho(y)(\rho(x)\delta_V) \big).
	\end{equation*}
	Then $\delta_V\in \mathrm{Der(gl}(V))$ if and only if 
	\begin{equation*}
		\rho(x)\rho(\delta y)+\rho(x)(\rho(y)\delta_V)=\rho(y)\rho(\delta x)\rho(y)(\rho(x)\delta_V) .
	\end{equation*}
	So
	\begin{equation*}
		\delta_V \in \mathrm{Der(gl}(V)) \Leftrightarrow \rho(x)(\delta_V \rho(y))=\rho(y)(\delta_V \rho(x)).
	\end{equation*}
	For the second identity, let $x,y \in L$ and suppose that $\delta_V \in \mathrm{InvDer(gl(V))}$. Then, we have
	\begin{equation*}
		\delta_V^2[\rho(x),\rho(y)]=[\delta_V\rho(x),\delta_V\rho(y)].
	\end{equation*}
	By linearity of $\delta_V$ and with simple computing we obtain 
	\begin{equation*}
		[\delta_V\rho(x),\delta_V\rho(y)]=\delta_V^2[\rho(x),\rho(y)]+\rho(x)(\delta_V^2\rho(y))-\rho(y)(\delta_V^2\rho(x)),
	\end{equation*}
which means 
	\begin{equation*}
		\rho(x)(\delta_V^2\rho(y))=\rho(y)(\delta_V^2\rho(x)).
	\end{equation*}
	For the opposite sense it is easy to verify. This completes the proof.
	
\end{proof}
\begin{prop}
	Let $(V;\rho,\delta_V)$ be a representation of the InvDer Lie algebra $(\mathfrak{L},\delta)$ and suppose that $\delta_V \in \mathrm{InvDer(gl}(V))$. Then for all $x,y \in L$, $(\mathrm{gl}(V),[\cdot,\cdot],\delta_V)$ is an InvDer-Lie algebra. 
\end{prop}
\begin{proof}
	Let $x,y,z \in L$. It is sufficient to prove the InvDer-Jacobi identity. Using \eqref{eq Der gl 1} and the fact that $(\mathrm{gl}(V),[\cdot,\cdot])$ is a Lie algebra, we obtain
	\begin{align*}
		[\delta_V\rho(x),[\rho(y),\rho(z)]]&=\delta_V\rho(x)(\rho(y)\rho(z))-\delta_V\rho(x)(\rho(z)\rho(y))-(\rho(y)\rho(z))\delta_V\rho(x)+(\rho(z)\rho(y))\delta_V\rho(x) \\
		&=\delta_V[\rho(x),[\rho(y),\rho(z)]]-(\rho(y)\rho(x))\delta_V\rho(z)+(\rho(z)\rho(x))\delta_V\rho(y).
	\end{align*}
Similarly, we have 
    \begin{align*}
		[\delta_V\rho(y),[\rho(z),\rho(x)]]&=\delta_V[\rho(y),[\rho(z),\rho(x)]]-(\rho(z)\rho(x))\delta_V\rho(y)+(\rho(x)\rho(y))\delta_V\rho(z), \\
		[\delta_V\rho(z),[\rho(x),\rho(y)]]&=\delta_V[\rho(z),[\rho(x),\rho(y)]]-(\rho(x)\rho(y))\delta_V\rho(z)+(\rho(y)\rho(x))\delta_V\rho(z).
	\end{align*}
With simple computing we obtain that 
	$\circlearrowleft_{(\rho(x),\rho(y),\rho(z))}[\delta_V\rho(x),[\rho(y),\rho(z)]]=0$. This completes the proof.
\end{proof}
\begin{ex}
	Let $(L;\mathrm{ad},\delta)$ be the adjoint representation of the InvDer Lie algebra $(\mathfrak{L},\delta)$. 
	Then $(\mathrm{gl(L),ad},\delta)$ is an InvDer Lie algebra.
\end{ex}
Next we are in position to define \textbf{the semidirect product of InvDer Lie algebra} $(\mathfrak{L},\delta)$ and a representation of it.
\begin{prop}
	Let $(V;\rho,\delta_V)$ be a representation of the InvDer Lie algebra $(\mathfrak{L},\delta)$. Define an anti-symmetric bilinear map 
	\begin{align*}
		[-,-]_{\ltimes}&:\wedge^2(L\oplus V) \rightarrow L\oplus V,
	\end{align*}
	by
	\begin{align*}
		[x+a,y+b]_{\ltimes}&=[x,y]_\mathrm{L}+\rho(x)(b)-\rho(y)(a) .
	\end{align*}
	We define also
	\begin{align*}
		\delta\oplus\delta_V&:L\oplus V \rightarrow L\oplus V,
	\end{align*}
	by
	\begin{align*}
		(\delta\oplus\delta_V)(x+a)&=\delta x+\delta_Va.
	\end{align*}
	Then $(L\oplus V,[-,-]_{\ltimes},\delta\oplus\delta_V)$ is an InvDer Lie algebra.
\end{prop}
\begin{proof}
	Let $x,y,z \in L$ and $a,b,c \in V$. The proof is organized in three steps as follow. The first step is dedicated to prove that $\delta\oplus\delta_V \in \mathrm{Der}(L\oplus V)$. We get
	\begin{align*}
		(\delta\oplus\delta_V)[x+a,y+b]_{\ltimes}&=(\delta\oplus\delta_V)([x,y]_L+\rho(x)(b)-\rho(y)(a)) \\
		&=\delta[x,y]_L+\delta_V\rho(x)(b)-\delta_V\rho(y)a   \\
		&\overset{\eqref{eq rep 3}}{=}[\delta x,y]_L+[x,\delta y]_L+\rho(\delta x)(b)+\rho(x)\delta_Vb-\rho(\delta y)(a)-\rho(y)\delta_Va  \\
		&=\Big([\delta x,y]_L+\rho(\delta x)(b)-\rho(y)\delta_Va \Big)+\Big([x,\delta y]_L+\rho(x)\delta_Vb-\rho(\delta y)(a) \Big) \\
		&=[(\delta\oplus\delta_V)(x+a),y+b]_{\ltimes}+[x+a,(\delta\oplus\delta_V)(y+b)]_{\ltimes}.
	\end{align*}
	It means that $\delta\oplus\delta_V \in \mathrm{Der}(L\oplus V)$. In the second step we prove that $\delta\oplus\delta_V \in \mathrm{InvDer}(L\oplus V)$.
	\begin{align*}
		(\delta\oplus\delta_V)^2[x+a,y+b]_\ltimes&=(\delta\oplus\delta_V)^2\Big([x,y]_L+\rho(x)(b)-\rho(y)(a) \Big) \\
		&=\delta^2[x,y]_L+\delta_V^2\rho(x)(b)-\delta_V^2(y)(a)\\
		&\overset{\eqref{eq rep 1}}{=}[\delta x,\delta y]_L+\rho(\delta x)\delta_Vb-\rho(\delta y)\delta_Va \\
		&=[\delta x+\delta_Va,\delta y+\delta_Vb]_{\ltimes} \\
		&=[(\delta\oplus\delta_V)(x+a),(\delta\oplus\delta_V)(y+b)]_{\ltimes}.
	\end{align*}
	In the last step, we prove the InvDer-Jacobi identity
	\begin{align*}
		[(\delta\oplus\delta_V)(x+a),[(y+b),(z+c)]_{\ltimes}]_{\ltimes}&=[(\delta x+\delta_V a),[x,y]_L+\rho(y)(c)-\rho(z)(b)]_{\ltimes} \\
		&=[\delta x,[y,z]_L]_L+\rho(\delta x)(\rho(y)(c)-\rho(z)(b))-\rho([y,z]_L)\delta_Va.
	\end{align*}
	Similarly, we obtain
	\begin{equation*}
		[(\delta\oplus\delta_V)(y+b),[z+c,x+a]_{\ltimes}]_{\ltimes}=[\delta y,[z,x]_L]_L+\rho(\delta y)(\rho(z)(a)-\rho(x)(c))-\rho([z,x]_L)\delta_Vb,
	\end{equation*}
	\begin{equation*}
		[(\delta\oplus\delta_V)(z+c),[x+a,y+b]_{\ltimes}]_{\ltimes}=[\delta z,[x,y]_L]_L+\rho(\delta z)(\rho(x)(b)-\rho(y)(b))-\rho([x,y]_L)\delta_Vc,
	\end{equation*}
	which leads us to 
	\begin{align*}
		\circlearrowleft_{(x,a),(y,b),(z,c)}[(\delta+\delta_V)(x,a),[(y,b),(z,c)]_{\ltimes}]_{\ltimes}&=\Big(\circlearrowleft_{x,y,z}[\delta x,[y,z]_L]_L+\rho(\delta x)\rho(y)(c) -\rho(\delta x)\rho(z)(b)\\
		&\ \ \ -\rho[y,z]_L\delta_Va+\rho(\delta y)(\rho(z)(a)-\rho(\delta y)\rho(x)(c)-\rho[z,x]_L\delta_Vb \\
		&\ \ \ +\rho(\delta z)(\rho(x)(b)-\rho(\delta z)\rho(y)(a)-\rho[x,y]_L\delta_Va \Big) \\
		&=0.
	\end{align*}
	This completes the proof.
\end{proof}
\begin{defn}
	A linear map $D:L \rightarrow L$ is called \textbf{$\delta$-derivation of the InvDer Lie algebra} $(L,[-,-]_\mathrm{L},\delta)$ if 
	\begin{align}
		D\circ \delta&=\delta^2 \circ D,\label{eq3.7} \\
		D[x,y]_\mathrm{L}&=[Dx,\delta y]_\mathrm{L}+[\delta x,Dy]_\mathrm{L}, \ \ \ \forall x,y \in L.  \label{eq3.8}
	\end{align}
	Denote the set of all $\delta$-derivation of the InvDer Lie algebra $(L,[-,-]_\mathrm{L},\delta)$ by $\mathrm{Der}_{\delta}(L)$.
\end{defn}

\begin{prop}
	For any $x\in L$ such that $\delta (x)=x$, the linear map 
	\begin{equation*}
	\left\{
	\begin{array}{cc}
	&\hspace{-2cm}\mathrm{ad}_x:L \rightarrow L,\\
	&\hspace{-.2cm}y\mapsto \mathrm{ad}_x(y)=[y,x]_\mathrm{L},
	\end{array}
	\right.
	\end{equation*}
	is a $\delta$-derivation of the InvDer Lie algebra $(\mathfrak{L},\delta)$. \\
\end{prop}
\begin{proof}
	Let $x,y,z\in L$ such that $\delta(x)=x$. Then, we have 
	\begin{align*}
		\mathrm{ad_x}(\delta y)&=[\delta y,x]_\mathrm{L}=[\delta y,\delta x]_\mathrm{L}=\delta^2[y,x]_\mathrm{L}=\delta^2\circ \mathrm{ad_x}(y).
	\end{align*}
	For the second assertion, we have 
	\begin{align*}
		\mathrm{ad_x}([y,z]_\mathrm{L})&=[[y,z]_\mathrm{L},\delta x]_\mathrm{L}=[[y,x]_\mathrm{L},\delta z]_\mathrm{L}+[\delta y,[z,x]_\mathrm{L}]_\mathrm{L}=[\mathrm{ad_x}(y);\delta z]_\mathrm{L}+[\delta y,\mathrm{ad_x}(z)]_\mathrm{L}.
	\end{align*}
	This means that $\mathrm{ad_x}$ is a $\delta$-derivation on $(\mathfrak{L},\delta)$.
\end{proof}
\section{Cohomology of InvDer Lie algebras}
Since the InvDer Lie algebraic structure has been obtained by twisting a Lie algebra $(\mathrm{L,[-,-]_L})$ by an Inv-derivation $\delta$ and inspired by the cohomology for Hom-Lie algebra of Lie-Type (see \cite{H0} section 2.1) we are, in this section, studying its cohomology similar to the context of the InvDer Lie algebraic structure.

First recall that the Chevalley-Eilenberg cohomology of the Lie algebra $\mathfrak{L}$ with coefficients in a representation $\mathcal{V}$ is the cohomology of the cochain complex $C^n(L;V)=\mathrm{Hom}(\wedge^nL,V)$ with the coboundary operator $\partial_\mathrm{CE}:C^n(L;V)\rightarrow C^{n+1}(L;V)$ defined by 
\begin{align*}
	\delta_\mathrm{CE}^n(f)(x_1,\cdots,x_{n+1})
	=&\displaystyle\sum_{i=1}^{n+1}(-1)^{i+1}\rho(x_i)f(x_1,\cdots,\hat{a}_i,\cdots,x_{n+1})\\
	&+\displaystyle\sum_{1\leq i<j\leq n+1}(-1)^{i+j}f([x_i,x_j]_\mathrm{L},x_1,\cdots,\hat{x}_i,\cdots,\hat{x}_j,\cdots,x_{n+1}).
\end{align*}
Let $\mathfrak{L}$ be a Lie algebra and $\mathcal{V}$ be a representation of it, let $\delta:L\rightarrow L$ be an Inv-derivation and $\delta_\mathrm{V}:V\rightarrow V$ an invertible linear map satisfying the equation \eqref{eq rep 1}. In this section we consider only $n\in\{1,2\}$
and we define the set of InvDer Lie algebra $n$-cochains by 
\begin{align*}
	\mathfrak{C}^1_\mathrm{InvDer}(L;V)&:=C^1(L;V)=\mathrm{Hom}(L,V),\\
	\mathfrak{C}^2_\mathrm{InvDer}(L;V)&:=(C^2(L;V)\oplus C^1(L;V))\oplus C^1(L;V)=(\mathrm{Hom}(\wedge^2L,V)\oplus\mathrm{Hom}(L,V))\oplus\mathrm{Hom}(L,V).
\end{align*}
We define the following operators:
\begin{equation*}
\hspace{-3.7cm}\left\{
\begin{array}{cc}
	&\Delta^1:C^1(L;V)\rightarrow C^1(L;V),\\
	&\hspace{-1cm}\Delta^1(f)=f\circ\delta-\delta_\mathrm{V}\circ f,
	\end{array}
	\right.
\left\{
\begin{array}{cc}
&\hspace{-1.5cm}\Delta^2:C^2(L;V)\rightarrow C^2(L;V),\\
&\Delta^2(f)=f(\delta\otimes \mathrm{Id})+f(\mathrm{Id}\otimes \delta)-\delta_\mathrm{V}\circ f,
\end{array}
\right.
\end{equation*}
\begin{equation*}
\left\{
\begin{array}{cc}
&\hspace{-1cm}\Delta^2_\mathrm{Inv}:C^2(L;V)\rightarrow C^2(L;V),\\
&\Delta_\mathrm{Inv}^2(f):=f(\delta\otimes\delta)-\delta_\mathrm{V}^2\circ f,
\end{array}
\right.
\left\{
\begin{array}{cc}
	&\hspace{-5cm}\phi^1:C^1(L;V)\rightarrow C^2(L;V),\\
	&\hspace{-0.5cm}\phi^1(f)(x,y):=\delta_\mathrm{V}f([x,y]_\mathrm{L})+f(\delta[x,y]_\mathrm{L})-\rho(\delta x)f(y)+\rho(\delta y)f(x).
	\end{array}
	\right.
\end{equation*}
We define also
\begin{equation*}
\hspace{-3cm}\left\{
\begin{array}{cc}
&\mathfrak{D}^1_\mathrm{InvDer}:\mathfrak{C}^1_\mathrm{InvDer}(L;V)\rightarrow \mathfrak{C}^2_\mathrm{InvDer}(L;V),\\
&\mathfrak{D}^1_\mathrm{InvDer}(f):=(\partial^1_\mathrm{CE}(f),-\Delta^1(f),-\Delta^1_\mathrm{Inv}(f)),
\end{array}
\right.
\end{equation*}
\begin{equation*}
\left\{
\begin{array}{cc}
&\hspace{-2cm}\mathfrak{D}^2_\mathrm{InvDer}:\mathfrak{C}^2_\mathrm{InvDer}(L;V)\rightarrow \mathfrak{C}^3_\mathrm{InvDer}(L;V),\\
&\mathfrak{D}^2_\mathrm{InvDer}(f,g,h):=(\partial^2_\mathrm{CE}(f),\partial^1_\mathrm{CE}(g)+\Delta^2(f),\phi^1(h)-\Delta^2_\mathrm{Inv}(f)).
\end{array}
\right.
\end{equation*}
\begin{defn}
	The space of $1$-cocycle of the InvDer Lie algebra $(\mathfrak{L},\delta)$ is defined
	\begin{equation*}
		\mathcal{H}^1_\mathrm{InvDer}(L;V):=\{f\in\mathfrak{C}^1_\mathrm{InvDer}(L;V);\quad \mathfrak{D}^1_\mathrm{InvDer}(f)=0\}.
	\end{equation*}
Also, the space of $2$-coboundaries of the InvDer Lie algebra $(\mathfrak{L},\delta)$ is 
\begin{equation*}
	\mathcal{B}^2_\mathrm{InvDer}(L;V):=\{(f,g,h)\in\mathfrak{C}^2_\mathrm{InvDer}(L;V);\quad \mathfrak{D}^1_\mathrm{InvDer}(\gamma)=(f,g,h) \text{ where } \gamma\in\mathfrak{C}^1_\mathrm{InvDer}(L;V)\}.
\end{equation*}
Moreover, the space of $2$-cocycles of the InvDer Lie algebra $(\mathfrak{L},\delta)$ is 
\begin{equation*}
	\mathcal{Z}^2_\mathrm{InvDer}(L;V):=\{(f,g,h)\in\mathfrak{C}^2_\mathrm{InvDer}(L;V);\quad \mathfrak{D}^2_\mathrm{InvDer}(f,g,h)=0\}.
\end{equation*}
\end{defn}
We need the following result in the proof of theorem \ref{thm coboundary}.
\begin{prop}
	For $f\in\mathfrak{C}^1_\mathrm{InvDer}(L;V)$, we have 
	\begin{equation}\label{derivation coboundary1}
		(\phi^1\circ\Delta^1_\mathrm{Inv}+\Delta^2_\mathrm{Inv}\circ \partial^1_\mathrm{CE})(f)=0.
	\end{equation}
\end{prop}
\begin{proof}
Let	$f\in\mathfrak{C}^1_\mathrm{InvDer}(L;V)$ and $x,y\in L$. Using the equation \eqref{eq rep 1} we obtain 
	\begin{align*}
		&(\Delta^2_\mathrm{Inv}\circ\partial^1_\mathrm{CE})(f)(x,y)\\
		=&\partial^1_\mathrm{CE}(f)(\delta x,\delta y)-\delta^2_V\circ\partial^1_\mathrm{CE}(f)(x,y)\\
		=&\rho(\delta x)f(\delta y)-\rho(\delta y)f(\delta x)-f(\delta^2[x,y]_\mathrm{L})\\
		&-\delta^2_V\rho(x)f(y)+\delta_V^2\rho(y)f(x)+\delta_V^2f([x,y]_\mathrm{L}).	
	\end{align*}
We get also
	\begin{align*}
		&(\phi^1\circ\Delta^1_\mathrm{Inv})(f)(x,y)\\
		=&\delta_V\circ\Delta^1_\mathrm{Inv}f([x,y]_\mathrm{L})+\Delta^1_\mathrm{Inv}f(\delta[x,y]_\mathrm{L})-\rho(\delta x)\Delta^1_\mathrm{Inv}(f)(y)+\rho(\delta y)\Delta^1_\mathrm{Inv}(f)(x)\\
		=&\delta_Vf(\delta[x,y]_\mathrm{L})-\delta_V^2f([x,y]_\mathrm{L})+f(\delta^2[x,y]_\mathrm{L})-\delta_Vf(\delta[x,y]_\mathrm{L})\\
		&-\rho(\delta x)f(\delta y)+\rho(\delta x)\delta_Vf(y)+\rho(\delta y)f(\delta x)-\rho(\delta y)\delta_Vf(x).
	\end{align*}
	Then 
	\begin{align*}
		&(\Delta^2_\mathrm{Inv}\circ\partial_\mathrm{CE}+\phi^1\circ\Delta^1_\mathrm{Inv})(f)(x,y)\\
		=&\rho(\delta x)f(\delta y)-\rho(\delta y)f(\delta x)-f(\delta^2[x,y]_\mathrm{L})\\
		&-\delta^2_V\rho(x)f(y)+\delta_V^2\rho(y)f(x)+\delta_V^2f([x,y]_\mathrm{L})\\
		&+\delta_Vf(\delta[x,y]_\mathrm{L})-\delta_V^2f([x,y]_\mathrm{L})+f(\delta^2[x,y]_\mathrm{L})-\delta_Vf(\delta[x,y]_\mathrm{L})\\
		&-\rho(\delta x)f(\delta y)+\rho(\delta x)\delta_Vf(y)+\rho(\delta y)f(\delta x)-\rho(\delta y)\delta_Vf(x)\\
		=&0.
	\end{align*}
\end{proof}
\begin{re}
	According to Lemma 3.1 of \cite{T0}, we have 
	\begin{equation}\label{derivation coboundary2}
		(\Delta^2\circ\partial^1_\mathrm{CE}-\partial^1_\mathrm{CE}\circ\Delta^1)(f)=0.
	\end{equation}
\end{re}
\begin{theorem}\label{thm coboundary}
	The map $\mathfrak{D}_\mathrm{InvDer}$ is a coboundary operator, i.e. $\mathfrak{D}^2_\mathrm{InvDer}\circ\mathfrak{D}^1_\mathrm{InvDer}=0$
\end{theorem}
\begin{proof}
	Let $f\in \mathfrak{C}^1_\mathrm{InvDer}(L;V)$. Using equations \eqref{derivation coboundary1} and \eqref{derivation coboundary2} and the fact that $\partial_\mathrm{CE}$ is a coboundary operator we obtain the following
	\begin{align*}
		\mathfrak{D}^2_\mathrm{InvDer}(\mathfrak{D}^1_\mathrm{InvDer})(f)&=\mathfrak{D}^2_\mathrm{InvDer}(\partial^1_\mathrm{CE}(f),-\Delta^1(f),-\Delta^1_\mathrm{Inv}(f))\\
		&=(\partial^2_\mathrm{CE}(\partial^1_\mathrm{CE})(f),-\partial^1_\mathrm{CE}(\Delta^1(f))+\Delta^2(\partial^1_\mathrm{CE}(f)),-\phi^1(\Delta^1_\mathrm{Inv}(f))-\Delta^2_\mathrm{Inv}(\partial^1_\mathrm{CE}(f)))\\
		&=0.
	\end{align*}
This completes the proof.
\end{proof}
\begin{defn}
	 Th quotient 
	\begin{equation*}
		\mathcal{H}^2_\mathrm{InvDer}(L;V)=\frac{\mathcal{Z}^2_\mathrm{InvDer}(L;V)}{\mathcal{B}^2_\mathrm{InvDer}(L;V)},
	\end{equation*}
	is called the second cohomology group of the InvDer Lie algebra $(\mathfrak{L},\delta)$.
\end{defn}
\begin{re}
	If $\delta$ is not an Inv-Derivation, means that it is a derivation on $\mathfrak{L}$ then the cohomology of InvDer Lie algebra reduces to the simply cohomology of LieDer pair $(\mathfrak{L},\delta)$ see \cite{T0} for more details.
\end{re}
\section{Formal one-parameter deformation of InvDer Lie algebra}
In this section, we discuss how the cohomology of InvDer Lie algebra can be used to study one-parameter formal deformations of an InvDer Lie algebra $(L,[-,-]_\mathrm{L},\delta)$, where the bracket $[-,-]_\mathrm{L}$ and the structure derivation $\delta$ are deformed. We use $\mathbb{K}[[t]]$ to denote the ring of formal power series over $K$ in a formal parameter $t$. For a vector space $L$, we use $L[[t]]$ to denote the space of formal power series of the form $\displaystyle\sum_{i\geq0}x_it^i$ for $x_i\in L$. We extend to InvDer Lie algebra the formal $1$-parameter formal deformation for a Lie algebra, means that we define the concept of $1$-parameter formal deformation for InvDer Lie algebra.
\begin{defn}
	Let $(\mathrm{L},[-,-]_\mathrm{L}=\mu(-,-),\delta)$ be an InvDer Lie algebra. A one-parameter formal deformation of $L$ is a couple $(\mu_t, \delta_t)$ contains the $\mathbb{K}[[\mathrm{t}]]$-bilinear map and $\mathbb{K}[[\mathrm{t}]]$-linear map
	\begin{equation*}
		\mu_t:L[[t]]\times L[[t]]\rightarrow L[[t]],\quad \delta_t : L[[t]] \to L[[t]],
	\end{equation*}
	of the forms
	
	\begin{equation*}
		\mu_t=\displaystyle\sum_{i\geq 0}\mu_i t^i,\quad \delta_\mathrm{t}=\displaystyle\sum_{i\geq 0}\delta_i t^i,
	\end{equation*}
	where each $\mu_i : L\times L\to L$ is a $\mathbb{K}$-bilinear map, and $\delta_i : L\to L$ is a $\mathbb{K}$-linear map satisfying
	\begin{eqnarray*}
		&&\circlearrowleft_{x,y,z}\mu_\mathrm{t}( x,\mu_\mathrm{t}(y,z))=0,\\
		&&\delta_t\mu_t(x,y)-\mu_t(\delta_tx,y)-\mu_t(x,\delta_ty)=0,\\
		&&(\delta_t \delta_t) \mu_\mathrm{t}(x,y)-\mu_\mathrm{t}(\delta_t x,\delta_t y)=0,
	\end{eqnarray*} 
	with $\mu_0=[-,-]_\mathrm{L},~\delta_0=\delta$. 
\end{defn}

\begin{re}
Notice that in the preceding definition of formal deformation, we omitted the InvDer-Jacobi identity for $\mu_t$ and $\delta_t$. This omission is a result of Proposition 2.7 in \cite{B0}. The InvDer-Jacobi identity is a consequence of the Jacobi identity of the specified Lie algebra and an invertible derivation, such that its inverse is also a derivation. Consequently, the InvDer-Jacobi identity will emerge from the stated conditions for $\mu_t$ and $\delta_t$.
\end{re}
For all $n\geq 0$, we have the following equations: 
\begin{align*}
	&\circlearrowleft_{x,y,z}\displaystyle\sum_{i+j=n}\mu_i( x,\mu_j(y,z))=0,\\
	&\sum_{i+j=n}\delta_i\mu_j(x,y)-\mu_i(\delta_jx,y)-\mu_i(x,\delta_jy)=0,\\
	&\sum_{i+j+k=n}\delta_i\delta_j \mu_\mathrm{k}(x,y)-\sum_{i+j+k=n} \mu_\mathrm{i}(\delta_j x,\delta_k y)=0.
\end{align*}
For $\mathrm{n=0}$, we deduce that $(\mathfrak{L},\delta)$ is InvDer Lie algebra. For $\mathrm{n=1}$, we have the following equations:
\begin{align}
	& \circlearrowleft_{x,y,z}\displaystyle [ x,\mu_1(y,z)]_\mathrm{L}+ \circlearrowleft_{x,y,z} \mu_1( x,[y,z]_\mathrm{L})=0,\label{equation def cocycle1}\\
	&\delta_1[x,y]_\mathrm{L}+\delta\mu_1(x,y)-\mu_1(\delta x,y)-[\delta_1x,y]_\mathrm{L}-\mu_1(x,\delta y)-[x,\delta_1y]_\mathrm{L}=0,\label{cocycle def 3}\\
	&\delta_1\delta [x,y]_\mathrm{L}+ \delta\delta_1 [x,y]_\mathrm{L}+ \delta\delta \mu_1(x,y)-\mu_1(\delta x,\delta y))- [\delta_1 x,\delta y]_\mathrm{L}- [\delta x,\delta_1 y]_\mathrm{L}=0\label{equation def cocycle2}.
\end{align}
Equation \eqref{equation def cocycle1} is equivalent to
\begin{equation*}
	\begin{split}
		&[ x,\mu_1(y,z)]_\mathrm{L}+[ y,\mu_1(z,x)]_\mathrm{L}+[ z,\mu_1(x,y)]+\mu_1( x,[y,z]_\mathrm{L})+\mu_1( y,[z,x]_\mathrm{L})+\mu_1(z,[x,y]_\mathrm{L}) =0,
	\end{split}
\end{equation*}
which is exactly 
\begin{equation}\label{deformation 1}
	\partial^2_\mathrm{CE}(\mu_1)(x,y,z)=0.
\end{equation}
The equation \eqref{cocycle def 3} is equivalent to 
\begin{equation*}
	-(\mu_1(\delta x,y)+\mu_1(x,\delta y)-\delta\mu_1(x,y))-([x,\delta_1y]_\mathrm{L}-[y,\delta_1x]_\mathrm{L}-\delta_1[x,y]_\mathrm{L})=0,
\end{equation*}
which is exactly 
\begin{equation}\label{deformation 2}
	(\partial^1_\mathrm{CE}(\delta_1)+\Delta(\mu_1))(x,y)=0.
\end{equation}
The equation \eqref{equation def cocycle2} is equivalent to 
\begin{equation*}
	\Big(\delta^2\mu_1(x,y)-\mu_1(\delta x,\delta y)\Big)+\Big(\delta\delta_1[x,y]_\mathrm{L}+\delta_1\delta[x,y]_\mathrm{L}-[\delta x,\delta_1y]_\mathrm{L}-[\delta_1x,\delta y]_\mathrm{L}\Big)=0,
\end{equation*}
which is exactly 
\begin{equation}\label{deformation 3}
	(\phi^1(\delta_1)-\Delta^2_\mathrm{Inv}(\mu_1))(x,y)=0.
\end{equation}
Therefore, according to equations \eqref{deformation 1}, \eqref{deformation 2} and \eqref{deformation 3}, we obtain 
\begin{equation*}
	(\partial^2_\mathrm{CE}(\mu_1),\partial^1_\mathrm{CE}(\delta_1)+\Delta(\mu_1),\phi^1(\delta_1)-\Delta^2_\mathrm{Inv}(\mu_1))=0,
\end{equation*}
means that 
\begin{equation*}
	\mathfrak{D}^2_\mathrm{InvDer}(\mu_1,\delta_1,\delta_1)=0.
\end{equation*}
In conclusion, $(\mu_1,\delta_1,\delta_1)\in\mathfrak{C}^2_\mathrm{InvDer}(L;V)$ is a $2$-cocycle of the InvDer Lie algebra $(\mathfrak{L},\delta)$ with coefficients in the adjoint representation.
Thus, from the above discussion, we have the following result:
\begin{theorem}
	Let $(\mathrm{\mu_t,\delta_t})$ be a one-parameter formal deformation of an InvDer Lie algebra $(\mathrm{\mathfrak{L},\delta})$. Then $\mathrm{(\mu_1,\delta_1,\delta_1)}$ is a $2$-cocycle of the InvDer Lie algebra $(\mathfrak{L},\delta)$ with coefficients in the adjoint representation.
\end{theorem}
\begin{defn}
	The $2$-cocycle $(\mu_1,\delta_1,\delta_1)$ is called the infinitesimal of the formal one parameter deformation $(\mu_t,\delta_t)$ of the InvDer Lie algebra $(\mathfrak{L},\delta)$.
\end{defn}
\begin{defn}
	Let $(\mu_t,\delta_t)$ and $(\mu^\prime_t,\delta^\prime_t)$ be two formal one-parameter deformations of an InvDer Lie algebra $(\mathfrak{L},\delta)$. A formal isomorphism between these two deformations is a power series
	\begin{equation*}
		\psi_t=\displaystyle\sum_{i=0}^{\infty}\psi_it^i:L[[t]]\rightarrow L[[t]],
	\end{equation*} 
where $\psi_i:L[[t]]\rightarrow L[[t]]$ are linear maps and $\psi_0=\mathrm{id_L}$ such that 
\begin{align*}
	\psi_t\circ\mu^\prime_t&=\mu_t\circ(\psi_t\otimes\psi_t),\\
	\psi_t\circ \delta^\prime_t&=\delta_t\circ\psi_t.
\end{align*}
Now expanding previous two equations and equating the coefficients of $t^n$ from both sides and putting $n=1$ we get 
\begin{align*}
	\mu_1^\prime(x,y)&=\mu_1(x,y)+\mu(\psi_1(x),y)+\mu(x,\psi_1(y))-\psi_1(\mu(x,y)),\\
	\delta^\prime_1x&=\delta_1x+\delta(\psi_1(x))-\psi(\delta x).
\end{align*}
Then, we have 
\begin{equation*}
	(\mu_1^\prime,\delta_1^\prime,\delta_1^\prime)-(\mu_1,\delta_1,\delta_1)=(\partial^1_\mathrm{CE}(\psi_1),-\Delta^1(\psi_1),-\Delta^1_\mathrm{Inv}(\psi_1))=\mathfrak{D}^2_\mathrm{InvDer}(\psi_1)\in \mathfrak{C}^1_\mathrm{InvDer}(L,L).
\end{equation*}
\end{defn}
\section{Central extension of InvDer Lie algebras}
\begin{defn}
	Let $(\mathcal{V},\delta_V)$ be an abelian InvDer Lie algebra and $(\mathfrak{L},\delta)$ be an InvDer Lie algebra. An exact sequence of InvDer Lie algebra morphisms
			\begin{center}
		$\xymatrix{
			0 \ar[r] & \mathcal{V} \ar[r]^i \ar[d]^{\delta_V} & \widehat{\mathfrak{L}} \ar[r]^p \ar[d]^{\widehat{\delta}} & \mathfrak{L} \ar[r] \ar[d]^{\delta} &0\\
			0 \ar[r] & \mathcal{V} \ar[r]^i & \widehat{\mathfrak{L}} \ar[r]^p & \mathfrak{L} \ar[r] &0
		}$
	\end{center}
is called a central extension of $(\mathfrak{L},\delta)$ by $(\mathcal{V},\delta_V)$ if $[V,\widehat{L}]_{\widehat{\mathfrak{L}}}=0$. Here, we identify $\mathcal{V}$ with the corresponding subalgebra of $\widehat{\mathfrak{L}}$, means that $\widehat{\delta}\vert_V=\delta_V$.
\end{defn}
\begin{defn}\label{definition iso extension}
	Let $(\widehat{\mathfrak{L}}_1,\widehat{\delta}_1)$ and $(\widehat{\mathfrak{L}}_2,\widehat{\delta}_2)$ be two central extensions of $(\mathfrak{L},\delta)$ by $(V,\delta_V)$. They are said to be isomorphic if there exists an InvDer Lie algebra morphism 
	\begin{equation*}
		\varphi:(\widehat{\mathfrak{L}}_1,\widehat{\delta}_1)\rightarrow (\widehat{\mathfrak{L}}_2,\widehat{\delta}_2),
	\end{equation*}
such that the following diagram is commutative
	$$\begin{CD}
	0@>>> (\mathcal{V},\delta_V) @>\mathrm{i_1} >> (\widehat{\mathfrak{L}}_1,\widehat{\delta}_1) @>\mathrm{p_1} >> (\mathfrak{L},\delta) @>>>\mathrm{0}\\
	@. @| @V \mathrm{\varphi} VV @| @.\\
	0@>>> (\mathcal{V},\delta_V) @>\mathrm{i_2} >> (\widehat{\mathfrak{L}}_2,\widehat{\delta}_2) @>\mathrm{p_2} >> (\mathfrak{L},\delta) @>>>\mathrm{0}.
\end{CD}$$
\end{defn}
A section of central extension $(\widehat{\mathfrak{L}},\widehat{\delta})$ of $(\mathfrak{L},\delta)$ by $(V,\delta_V)$ is a linear map $s:L\rightarrow \widehat{L}$ such that 
\begin{equation*}
	p\circ s=\mathrm{Id}.
\end{equation*}
Let $(\widehat{\mathfrak{L}},\widehat{\delta})$ be central extension of an InvDer Lie algebra $(\mathfrak{L},\delta)$ by an abelian InvDer Lie algebra $(V,\delta_V)$ and $s:L\rightarrow \widehat{L}$ be a section of it. We define linear maps 
\begin{equation*}
	\gamma:L\wedge L\rightarrow V,\quad \chi:L\rightarrow V,
\end{equation*}
respectively by 
\begin{align*}
	\gamma(x,y)&=[s(x),s(y)]_{\widehat{\mathfrak{L}}}-s[x,y],\\
	\chi(x)&=\widehat{\delta}s(x)-s(\delta x), \forall x,y\in L.
\end{align*}
We have $\widehat{L}$ is isomorphic to $L\oplus V$ as a vector spaces. So we can define the InvDer Lie algebraic structure on $L\oplus V$ similarly to $\widehat{\mathfrak{L}}$ to obtain an InvDer Lie algebra $(L\oplus V,[-,-]_\gamma,\delta_\chi)$, where the InvDer Lie structure is given by 
\begin{align}
	[x+u,y+v]_\gamma&=[x,y]_\mathrm{L}+\gamma(x,y),\\
	\delta_\chi(x+u)&=\delta x+\chi(x)+\delta_Vu,\quad \forall x,y\in L, u,v\in V. 
\end{align}
\begin{prop}
	With the above notation, $(L\oplus V,[-,-]_\gamma,\delta_\chi)$ is an InvDer Lie algebra if and only if $(\gamma,\chi,\chi)$ is a $2$-cocycle of the InvDer Lie algebra $(\mathfrak{L},\delta)$ with coefficients in the trivial representation $(V;\rho=0,\delta_V)$ means that $(\gamma,\chi,\chi)$ satisfy the following equations 
	\begin{align}
		\gamma([x,y]_\mathrm{L},z)+\gamma([y,z]_\mathrm{L},x)+\gamma([z,x]_\mathrm{L},y)&=0,\label{extension 1}\\
		\chi([x,y]_\mathrm{L})+\delta_V(\gamma(x,y))-\gamma(\delta x,y)-\gamma(x,\delta y)&=0,\label{extension 2}\\
		\chi(\delta[x,y]_\mathrm{L})+\delta_V\chi([x,y]_\mathrm{L})+\delta_V^2\gamma(x,y)-\gamma(\delta x,\delta y)&=0,\label{extension 3}
	\end{align}
for all $x,y,z\in L$.
\end{prop}
\begin{proof}
	If $(L\oplus V,[-,-]_\gamma,\delta_\chi)$ is an InvDer Lie algebra, we have  
	\begin{equation*}
		[[x+u,y+v]_\gamma,\delta_\chi(z+t)]_\gamma+\mathrm{c.p}=0,\quad \forall x,y,z\in L \text{ and } \forall u,v,t\in V,
	\end{equation*}
which is equivalent to 
\begin{align}
	[[x+u,y+v]_\gamma,z+t]_\gamma+\mathrm{c.p}&=0,\label{extension 4}\\
	\delta_\chi[x+u,y+v]_\gamma&=[\delta_\chi(x+u),y+v]_\gamma+[x+u,\delta_\chi(y+v)]_\gamma,\label{extension 5}\\
	\delta^2_\chi[x+u,y+v]_\gamma&=[\delta_\chi(x+u),\delta_\chi(y+v)]_\gamma.\label{extension 6}
\end{align}
Now expanding equation \eqref{extension 4} we obtain 
\begin{equation*}
	[[x,y]_\mathrm{L},z]_\mathrm{L}+\mathrm{c.p}+\gamma([x,y]_\mathrm{L},z)+\gamma([y,z]_\mathrm{L},x)+\gamma([z,x]_\mathrm{L},y)=0.
\end{equation*}
By the Jacobi identity the equation \eqref{extension 1} holds. Similarly for the equation \eqref{extension 5} we obtain
\begin{equation*}
	\delta[x,y]_\mathrm{L}+\chi([x,y]_\mathrm{L})+\delta_V\gamma(x,y)=[\delta x,y]_\mathrm{L}+\gamma(\delta x,y)+[x,\delta y]_\mathrm{L}+\gamma(x,\delta y),
\end{equation*}
and using the fact that $\delta$ is a derivation on $\mathfrak{L}$ the equation \eqref{extension 2} holds. Finally, equation \eqref{extension 6} is equivalent to 
\begin{equation*}
	\delta^2[x,y]_\mathrm{L}+\chi(\delta[x,y]_\mathrm{L})+\delta_V\chi([x,y]_\mathrm{L})+\delta_V^2\gamma(x,y)=[\delta x,\delta y]_\mathrm{L}+\gamma(\delta x,\delta y),
\end{equation*}
and using the fact that $\delta\in \mathrm{InvDer(L)}$ the equation \eqref{extension 3} holds.
Conversely, if equations \eqref{extension 4}, \eqref{extension 5} and \eqref{extension 6} hold, then it is straightforward to see that $(L\oplus V,[-,-]_\gamma,\delta_\chi)$ is an InvDer Lie algebra. 
\end{proof}
\begin{theorem}
	Let $(\mathcal{V},\delta_V)$ be an abelian InvDer Lie algebra and $(\mathfrak{L},\delta)$ be an InvDer Lie algebra. Then central extension of $(\mathfrak{L},\delta)$ by $(\mathcal{V},\delta_V)$ is classified by the second cohomology group $\mathcal{H}_\mathrm{InvDer}^2(L;V)$ of the InvDer Lie algebra $(\mathfrak{L},\delta)$ with the coefficients in the trivial representation $(V;\rho=0,\delta_V)$. 
\end{theorem}
\begin{proof}
	Let $(\widehat{\mathfrak{L}},\widehat{\delta})$ be a central extension of $(\mathfrak{L},\delta)$ by $(\mathcal{V},\delta_V)$. By choosing a section $s:L\rightarrow \widehat{L}$, we obtain a $2$-cocycle $(\gamma,\chi,\chi)$. Now we prove that the cohomological class of $(\gamma,\chi,\chi)$ does not depend on the choice of sections. Let $s_1$ and $s_2$ be two different sections. We define $\Phi:L\rightarrow V$ by 
	\begin{equation*}
		\Phi(x)=s_1(x)-s_2(x),
	\end{equation*}
so, we have 
\begin{align*}
	\gamma_1(x,y)&=[s_1(x),s_2(x)]_{\widehat{\mathfrak{L}}}-s_1[x,y]_\mathrm{L}\\
	&=[s_2(x)+\Phi(x),s_2(y)+\Phi(y)]_{\widehat{\mathfrak{L}}}-s_2[x,y]_\mathrm{L}-\Phi([x,y]_\mathrm{L})\\
	&=[s_2(x),s_2(y)]-s_2[x,y]-\Phi([x,y])\\
	&=\gamma_2(x,y)-\Phi([x,y]_\mathrm{L}),
\end{align*}
and 
\begin{align*}
	\chi_1(x)&=\widehat{\delta}(s_1(x))-s_1(\delta x)\\
	&=\widehat{\delta}(s_2(x)+\Phi(x))-s_2(\delta x)-\Phi(\delta x)\\
	&=\widehat{\delta}(s_2(x))-s_2(\delta x)+\delta_V\Phi(x)-\Phi(\delta x)\\
	&=\chi_2(x)+\delta_V(\Phi(x))-\Phi(\delta x).
\end{align*}
Then, we obtain 
\begin{equation*}
	(\gamma_1,\chi_1,\chi_1)=(\gamma_2,\chi_2,\chi_2)+\mathfrak{D}^1_\mathrm{InvDer}(\Phi).
\end{equation*}
Therefore, $(\gamma_1,\chi_1,\chi_1)$ and $(\gamma_2,\chi_2,\chi_2)$ are in the same cohomological class. Finally, we show that isomorphic central extensions give rise to the same element in $\mathcal{H}^2_\mathrm{InvDer}(L;V)$. Suppose that $(\widehat{\mathfrak{L}}_1,\widehat{\delta}_1)$ and $(\widehat{\mathfrak{L}}_2,\widehat{\delta}_2)$ are two isomorphic central extensions of $(\mathfrak{L},\delta)$ by $(\mathcal{V},\delta_V)$ and $\xi:(\widehat{\mathfrak{L}}_1,\widehat{\delta}_1)\rightarrow (\widehat{\mathfrak{L}}_2,\widehat{\delta}_2)$ is an InvDer Lie algebra morphism such that we have the commutative diagram in definition \eqref{definition iso extension}. Assume that $s_1:L\rightarrow \widehat{L}$ is a section of $\widehat{\mathfrak{L}}_1$. By $p_2\circ \xi=p_1$, we have 
\begin{equation*}
	p_2=(\xi\circ s_1)=p_1\circ s_1=\mathrm{Id}.
\end{equation*}
Thus, $\xi\circ s_1$ is a section of $\widehat{\mathfrak{L}}_2$. Deifne $s_2=\xi\circ s_1$ and since $\xi$ is a morphism of InvDer Lie algebra and $\xi\vert_V=\mathrm{Id}$, we have 
\begin{align*}
	\gamma_2(x,y)&=[s_2(x),s_2(y)]_{\widehat{\mathfrak{L}}_2}-s_2[x,y]_\mathrm{L}\\
	&=[(\xi\circ s_1)(x),(\xi\circ s_2)(x)]_{\widehat{\mathfrak{L}}_2}-(\xi\circ s_1)[x,y]_\mathrm{L}\\
	&=\xi([s_1(x),s_1(y)]_{\widehat{\mathfrak{L}}_1}-s_1[x,y]_\mathrm{L})\\
	&=[s_1(x),s_1(y)]_{\widehat{\mathfrak{L}}_1}-s_1[x,y]_\mathrm{L}\\
	&=\gamma_1(x,y).
\end{align*}
Similarly, we show that $\chi_2(x)=\chi_1(x)$.
Thus, the isomorphic central extensions give rise to the same element in $\mathcal{H}_\mathrm{InvDer}^2(L;V)$. Conversely, giving two $2$-cocycles $(\gamma_1,\chi_1,\chi_1)$ and $(\gamma_2,\chi_2,\chi_2)$, we construct two central extensions $(L\oplus V,[-,-]_{\gamma_1},\delta_{\chi_1})$ and $(L\oplus V,[-,-]_{\gamma_2},\delta_{\chi_2})$ as in \eqref{extension 1}, \eqref{extension 2} and \eqref{extension 3}. If they represent the same cohomological classes, i.e., there exists $\Phi:L\rightarrow V$ such that 
\begin{equation*}
	(\gamma_1,\chi_1,\chi_1)=(\gamma_2,\chi_2,\chi_2)+\mathfrak{D}^1_\mathrm{InvDer}(\Phi),
\end{equation*}
then defining $\xi:L\oplus V\rightarrow L\oplus V$ by
\begin{equation*}
	\xi(x+u):=x+\Phi(x)+u,
\end{equation*}
we can deduce that $\xi$ is an isomorphism between central extensions. This completes the proof.
\end{proof}

\noindent {\bf Acknowledgment:}
The authors would like to thank the referee for valuable comments and suggestions on this article.  Ripan Saha is supported by the Core Research Grant (CRG) of Science and Engineering Research Board (SERB), Department of Science and Technology (DST), Govt. of India (Grant Number- CRG/2022/005332).


\end{document}